\UseRawInputEncoding
\documentclass[12pt]{article}

\usepackage{dsfont}
\usepackage{amsfonts,amsmath,amsthm}
\usepackage{algorithm, algorithmic}
\usepackage{color}
\usepackage{graphicx}
\usepackage{stmaryrd}        

\usepackage{multicol}
\usepackage{multirow}

\usepackage[paper=a4paper,dvips,top=2cm,left=2cm,right=2cm,
foot=1cm,bottom=4cm]{geometry}

\newcommand{\red}[1]{\begin{color}{black}#1\end{color}}
\newcommand{\ve}{{\bf e}}

\begin{document}
	\large
	
	\title{Dual Number Matrices with Primitive and Irreducible Nonnegative Standard Parts}
	\author{ Liqun Qi\footnote{Department of Mathematics, School of Science, Hangzhou Dianzi University, Hangzhou 310018 China; Department of Applied Mathematics, The Hong Kong Polytechnic University, Hung Hom, Kowloon, Hong Kong
			({\tt maqilq@polyu.edu.hk}).}
		\and \
		Chunfeng Cui\footnote{LMIB of the Ministry of Education, School of Mathematical Sciences, Beihang University, Beijing 100191 China.
			({\tt chungfengcui@buaa.edu.cn}).}
	}
	\date{\today}
	\maketitle

	\begin{abstract}
		In this paper, we extend the Perron-Frobenius theory to dual number matrices with primitive and irreducible nonnegative standard parts.  One motivation of our research is to consider probabilities as well as perturbation, or error bounds, or variances, in the Markov chain process.  We show that  such a dual number matrix always has a positive dual number eigenvalue with a positive dual number eigenvector.   The standard part of this positive dual number eigenvalue is larger than or equal to the modulus of the standard part of any other eigenvalue of this dual number matrix.  We present an explicit formula to compute the dual part of this positive dual number eigenvalue.  The Collatz minimax theorem also holds here.
		\red{The results are nontrivial as even a positive dual number matrix may have no eigenvalue at all.}
		An algorithm based upon the Collatz minimax {theorem} is constructed. The convergence of the algorithm is studied.
		\red{We show the upper bounds  on the distance of stationary states
			between the dual Markov chain and the perturbed Markov chain.}
		Numerical results {on both  synthetic examples and dual Markov chain}
		{including some real world examples} are reported.

		\medskip


		
	\end{abstract}

	\renewcommand{\Re}{\mathds{R}}
	\newcommand{\rank}{\mathrm{rank}}
	\newcommand{\X}{\mathcal{X}}
	\newcommand{\A}{\mathcal{A}}
	\newcommand{\I}{\mathcal{I}}
	\newcommand{\B}{\mathcal{B}}
	\newcommand{\C}{\mathcal{C}}
	\newcommand{\OO}{\mathcal{O}}
	\newcommand{\e}{\mathbf{e}}
	\newcommand{\0}{\mathbf{0}}
	\newcommand{\dd}{\mathbf{d}}
	\newcommand{\ii}{\mathbf{i}}
	\newcommand{\jj}{\mathbf{j}}
	\newcommand{\kk}{\mathbf{k}}
	\newcommand{\va}{\mathbf{a}}
	\newcommand{\vb}{\mathbf{b}}
	\newcommand{\vc}{\mathbf{c}}
	\newcommand{\vq}{\mathbf{q}}
	\newcommand{\vg}{\mathbf{g}}
	\newcommand{\pr}{\vec{r}}
	\newcommand{\ps}{\vec{s}}
	\newcommand{\pt}{\vec{t}}
	\newcommand{\pu}{\vec{u}}
	\newcommand{\pv}{\vec{v}}
	\newcommand{\pw}{\vec{w}}
	\newcommand{\pp}{\vec{p}}
	\newcommand{\pq}{\vec{q}}
	\newcommand{\pl}{\vec{l}}
	\newcommand{\vt}{\rm{vec}}
	\newcommand{\vx}{\mathbf{x}}
	\newcommand{\vy}{\mathbf{y}}
	\newcommand{\vu}{\mathbf{u}}
	\newcommand{\vv}{\mathbf{v}}
	\newcommand{\y}{\mathbf{y}}
	\newcommand{\vz}{\mathbf{z}}
	\newcommand{\T}{\top}
	
	\newtheorem{Thm}{Theorem}[section]
	\newtheorem{Def}[Thm]{Definition}
	\newtheorem{Ass}[Thm]{Assumption}
	\newtheorem{Lem}[Thm]{Lemma}
	\newtheorem{Prop}[Thm]{Proposition}
	\newtheorem{Cor}[Thm]{Corollary}
	\newtheorem{example}[Thm]{Example}
	\newtheorem{remark}[Thm]{Remark}
	
	\section{Introduction}
	
	Dual numbers and dual number matrices have applications in kinematic synthesis, dynamic analysis of spatial mechanisms, robotics, etc., and have attracted wide {attention} 
	\cite{An12, GL87, PV09, PVDA16, QLY22, Wa21, WCW23}.
	

	For the Markov chain, {besides} the probability distribution, one may also need to consider the perturbation, or the error bound, or the variance, of the probability distribution.  Then we propose a dual Markov chain model to accommodate the perturbation, or the error bound, or the variance.  This motivates us to consider dual number matrices with nonnegative standard parts.
	
	
	Dual number matrices are the special cases of dual quaternion matrices.   By Qi and Luo \cite{QL23}, an $n \times n$ dual number symmetric matrix has exactly $n$ eigenvalue.   This matrix is positive definite if and only if these $n$ eigenvalues are nonnegative.
	
	However, the situation for eigenvalues of nonsymmetric dual number matrices is very bad.
	
	We may regard dual numbers as special cases of dual complex numbers.  There are two approaches to {considering} the eigenvalue theory of general dual complex matrices.  One approach is contained in \cite{QL21}.   In this approach, dual complex multiplication is defined as noncommutative.   The motivation for this is to represent rigid body motion in the plane.   
	As the multiplication is noncommutative, one only can define right and left eigenvalues as the quaternion matrix case.   
	We do not use this approach here.   Another approach is to consider dual complex numbers as a special case of dual quaternion numbers.   This approach is contained in \cite{QACLL22} and {has been} 
	further studied in \cite{QC23}. With this approach, eigenvalues of dual complex matrices were defined in \cite{QACLL22}.  In \cite{QC23}, several examples of dual number matrices were given. In one example a square  positive  dual number matrix has no eigenvalues at all, and in another example a {nonnegative} square dual number matrix has infinitely many eigenvalues.  It was shown that an $n \times n$  dual complex matrix has exactly $n$ eigenvalues with $n$ appreciably linearly independent eigenvectors if and only if it is similar to a diagonal matrix.  In this paper, we follow this approach to investigate dual number matrices with primitive and irreducible nonnegative standard parts.
	
	In the next section, we consider the Markov chain, which not only has the probability distribution, but also its perturbation, or error bound, or variance.  We call such a Markov chain a dual Markov chain.
	
	In Section 3, we review some preliminary knowledge of dual numbers, dual complex numbers, and eigenvalues of dual complex matrices.
	
	In Section 4, we show that if a square dual number matrix $A$ has a primitive standard part $A_s$, then $A$ has an eigenvalue $\lambda$, which is a positive dual number, larger than the modulus of any other eigenvalue of $A$, and has a positive dual number eigenvector.  We call this eigenvalue the Perron eigenvalue of $A$.  In particular, we present an explicit formula to compute the dual part of $\lambda$.  Then we show that the Collatz minimax theorem holds for this Perron eigenvalue.
	
	In Section 5, we show that if a square dual number matrix $A$ has an irreducible nonnegative standard part $A_s$, then $A$ has an eigenvalue $\lambda$, which is a positive dual number, and has a positive dual number eigenvector.  {We call this eigenvalue the Perron-Frobenius eigenvalue of $A$.  The Collatz minimax theorem still holds for this Perron-Frobenius eigenvalue.}
	The standard part of the Perron-Frobenius eigenvalue $\lambda$ is greater than or equal to the modulus of the standard part of any other eigenvalue of $A$.  If $A_s$ is not primitive, i.e., the period $h$ of $A_s$ is greater than one, then $A$ has $h-1$ other eigenvalues $\lambda_j$ for $j = 1, \dots, h-1$, such that the modulus of the standard part of $\lambda_j$ is equal to $\lambda_s$.  We give an example to show that in this case, the modulus of $\lambda_j$ may be greater than the modulus of $\lambda$ for some $j$, in the sense of \cite{QLY22}.
	
	Based upon the Collatz minimax theorem for dual number matrices with primitive and irreducible nonnegative standard parts, we present an iterative algorithm for computing the Perron-Frobenius  eigenvalue in Section 6.
	We also study the convergence properties of this algorithm there.
	
	In Section 7, we report the numerical results of this algorithm {by several synthetic examples}.
	{Under the condition that the dual transition matrix has a  primitive or irreducible nonnegative standard part, and  $A_s^\top {\bf 1}={\bf 1}$,  $A_d^\top {\bf 1}={\bf 0}$,   we show the Perron-Frobenius eigenvalue of the dual number  matrix $A$ is always 1.
		Furthermore, we give an upper bound on the  distance of stationary states   between the dual Markov chain and the perturbed Markov chain.}
	We also compare dual Markov chain and perturbed Markov chain numerically there. 
	
	Some final remarks are made in Section 8.
	
	\section{Dual Markov Chain}
	
	A Markov chain is a stochastic model describing a sequence of events in which the probability of each event depends only on the state attained in the previous event \cite{CHNS06}.  Consider a stochastic process $\{ X_t : t = 1, 2, 3, \dots \}$.   Suppose that $X_t$ takes values in $n$ possible states $\{ 1, \dots, n \}$, where $n$ is a positive integer.  We may describe the situation at time $t$ by an $n$ dimensional real nonnegative vector
	$$\vx^{(t)} = \left(\begin{array} {c} x_1^{(t)} \\ \vdots \\ x_j^{(t)} \\ \vdots \\ x_n^{(t)} \end{array}\right),$$
	where $x_j^{(t)} = \text{Prob}(X_t = j)$.   Then we have $\sum_{j=1}^n x_j^{(t)} = 1$ and $x_j^{(t)} \ge 0$ for $j = 1, \dots, n$ {and $t\ge 1$}.   We call $\vx^{(t)}$ the probability distribution at time $t$.
	
	In the Markov chain model, the probability of $X_t = i$ if $X_{t-1} = j$, is
	$p_{ij} = \text{Prob}(X_t = i | X_{t-1} = j).$
	Then we have an $n \times n$  nonnegative matrix $P = (p_{ij})$, satisfying $p_{ij} \ge 0$ for $i, j = 1, \dots, n$ and $\sum_{i=1}^n p_{ij} = 1$ for $j = 1, \dots, n$.   We call $P$ a transition probability matrix.  Denote $S_n =  \{ \vx \in {\mathbb R}^n : \vx \ge \0, \sum_{j=1}^n x_j = 1 \}$.  If we have $\vx^{(t-1)} \in S_n$, then $\vx^{(t)} = P\vx^{(t-1)} \in S_n$.   Assume that
	$$\lim_{t \to \infty} \vx^{(t)} = \vx^*.$$
	Then $\vx^* \in S_n$ is called the stationary probability distribution of the Markov chain, or the limiting probability distribution of the transition probability matrix $P$.  Then we have
	\begin{equation} \label{pM1}
		P\vx^* = \vx^*.
	\end{equation}
	This implies that $\vx^*$ is an eigenvector of $P$, corresponding to the eigenvalue $\lambda = 1$.  Because of the properties of $P$, according to the nonnegative matrix theory \cite{BP94}, $\lambda = 1$ is the largest eigenvalue of $P$, called the Perron-Frobeninus eigenvalue of $P$.  There is a rich theory on nonnegative matrices, with {the} Markov chain as one of its applications.
	
	However, in the real world, the data has {perturbations}, or errors, or variances.  We need to consider the perturbation, or the error bound, or the variances.  Hence, for each probability $x_j^{(t)}$ and $p_{ij}$, we have to use two real numbers to denote them.  We may denote
	$x_j^{(t)} = \left(x_{sj}^{(t)}, x_{dj}^{(t)}\right)$.  Here, $x_{sj}^{(t)}$ is the probability, and $x_{dj}^{(t)}$ is its perturbation, or error bound, or variance.  The meanings of the two letters ``s'' and ``d'' will be clarified later.    We have $\vx^{(t)} = \left( \vx_s^{(t)}, \vx_d^{(t)} \right)$, where
	$$\vx_s^{(t)} = \left(\begin{array} {c} x_{s1}^{(t)} \\ \vdots \\ x_{sj}^{(t)} \\ \vdots \\ x_{sn}^{(t)} \end{array}\right) \ \ {\rm and} \ \ \vx_d^{(t)} = \left(\begin{array} {c} x_{d1}^{(t)} \\ \vdots \\ x_{dj}^{(t)} \\ \vdots \\ x_{dn}^{(t)} \end{array}\right).$$
	Then $\vx_s^{(t)} \in S_n$ and $\vx_d^{(t)}$ is a real vector or a nonnegative vector, depending {on} its physical meanings.   We call the two-part vector $\vx^{(t)}$ the {\bf dual probability distribution} at time $t$.   Similarly, we have $p_{ij} = \left(p_{sij}, p_{dij}\right)$.  Again, $p_{sij}$ is the probability, and $p_{dij}$ is its perturbation, or error bound, or variance.   We have $P = \left(P_s, P_d\right)$, where $P_s = \left( p_{sij}\right)$ is a nonnegative matrix, satisfying
	$\sum_{i=1}^n p_{sij} = 1$ for $j = 1, \dots, n$, and $P_d = \left( p_{dij}\right)$ is a real matrix or a nonnegative matrix, respectively.   We call the two-part matrix $P$ a {\bf dual transition probability matrix}.
	
	Now, we need to define the addition and multiplication of these two-part numbers.  Let $a = (a_s, a_d)$ and $b = (b_s, b_d)$ be two two-part numbers.  Here, $a_d$ and $b_d$ mean perturbations.  We may define $a+b = (a_s + b_s, a_d+ b_d)$.    For multiplication, we think that the product of perturbations should be neglected, i.e., we may think $a_db_d = 0$.  Then we may define $ab = (a_sb_s, a_sb_d+a_db_s)$.  It turns out that by such definitions of addition and multiplication, the two-part numbers are nothing else, but the dual numbers.  Then we may regard $\vx$ as a dual number vector, $\vx_s$ as its standard part, and $\vx_d$ as its dual part.  This explains the meanings of the two letters ``s'' and ``d''.   Similarly, we regard $P_s$ as the standard part of $P$, and $P_d$ as the dual part of $P$.  We still have $\vx^{(t)} = P\vx^{(t-1)}$.  Again, assume that
	$$\lim_{t \to \infty} \vx^{(t)} = \vx^*.$$
	Then $\vx^* = \left(\vx_s^*, \vx_d^*\right)$, $\vx^*_s \in S_n$,  and $\vx^*$ is called the {\bf dual stationary probability distribution} of the {\bf dual Markov chain}, or the {\bf dual limiting probability distribution} of the dual transition probability matrix $P$.  Then, what does the equation (\ref{pM1}) look like?
	Is $1$ still the largest eigenvalue of the dual number matrix $P$?   Or a dual number $\lambda$ is the largest eigenvalue of $P$ and the standard part of $\lambda$ is $1$.  To study more about the dual Markov chain, we have to develop the Perron-Frobenius theory for dual number matrices with
	nonnegative standard parts.  This motivated our research in this paper.

	\section{Dual Numbers, Dual Complex Numbers, and Eigenvalues of Dual Complex Matrices}

	\subsection{Dual Numbers and Dual Complex Numbers}
	
	The field of real numbers, the field of complex numbers, the set of dual numbers, and the set of dual complex numbers are denoted  by $\mathbb R$, $\mathbb C$, $\mathbb D$ and $\mathbb {DC}$, respectively.

	A {\bf dual complex number} $a = a_s + a_d\epsilon \in \mathbb {DC}$ has standard part $a_s$ and dual part $a_d$.   Both $a_s$ and $a_d$ are complex numbers.   The symbol $\epsilon$ is the infinitesimal unit, satisfying $\epsilon\neq 0$, $\epsilon^2 = 0$, and $\epsilon$ is commutative with complex numbers.  If $a_s \not = 0$, then we say that $a$ is {\bf appreciable}.    If $a_s$ and $a_d$ are real numbers, then $a$ is called a {\bf dual number}.
	
	The conjugate of $a = a_s + a_d\epsilon$ is $a^* = a_s^* + a_d^*\epsilon$.
	
	Suppose we have two dual complex numbers $a = a_s + a_d\epsilon$ and $b = b_s + b_d\epsilon$.   Then their sum is $a+b = (a_s+b_s) + (a_d+b_d)\epsilon$, and their product is $ab = a_sb_s + (a_sb_d+a_db_s)\epsilon$.
	The multiplication of dual complex numbers is commutative.
	
	Suppose we have two dual numbers $a = a_s + a_d\epsilon$ and $b = b_s + b_d\epsilon$.   By \cite{QLY22}, if
	$a_s > b_s$, or $a_s = b_s$ and $a_d > b_d$, then we say $a > b$.   Then this defines positive, nonnegative dual numbers, etc.
	Denote $\mathbb D_+$, $\mathbb D_{++}$ as the set of nonnegative  and positive dual numbers, respectively.
	In particular, for a dual number $a = a_s + a_d\epsilon$, its magnitude is defined as a nonnegative dual number
	$$|a| = \left\{ \begin{array}{ll} |a_s| + {\rm sgn}(a_s)a_d\epsilon, & \ {\rm if}\  a_s \not = 0, \\ |a_d|\epsilon, &   \ {\rm otherwise}.  \end{array}  \right.$$
	For any  dual   number $a=a_s+a_d\epsilon$ and dual number  $b=b_s+b_d\epsilon$ with $b_s\neq 0$, or $a_s=0$ and $b_s=0$, there is
	\begin{equation*}
		\frac{a_s+a_d\epsilon}{b_s+b_d\epsilon} =
		\left\{
		\begin{array}{ll}
			\frac{a_s}{b_s}+\left(  \frac{a_d}{b_s}- \frac{a_s}{b_s} \frac{b_d}{b_s}\right)\epsilon,   & \text{ if } b_s\neq 0, \\
			\frac{a_d}{b_d} +c\epsilon,  & \ \text{if } a_s= 0, b_s= 0,\\
		\end{array}
		\right.
	\end{equation*}
	where $c$ is an arbitrary   number.
	
	Later, when a dual number is   nonnegative or positive, we say it is a nonnegative dual number or a positive dual number respectively.   If we say a number is a nonnegative number or positive number, then that number should be a real number.
	We   use $0$, ${\bf 0}$, and $O$ to denote a zero number, a zero vector, and a zero matrix, respectively.
	
	A dual complex number vector is denoted by $\vx = (x_1, \cdots, x_n)^\top \in {\mathbb {DC}}^n$.  Its $2$-norm is defined as
	$$\|\vx\|_2 = \left\{\begin{array}{ll}
		{\|\vx_s\|_2+\frac{\vx_s^*\vx_d}{\|\vx_s\|_2}\epsilon,}
		& \ \mathrm{if} \  \vx_s \not = \0,\\
		\|\vx_d\|_2\epsilon, & \ \mathrm{if} \  \vx_s  = \0.
	\end{array}\right.$$
	We may denote $\vx = \vx_s + \vx_d\epsilon$, where $\vx_s, \vx_d \in {\mathbb C}^n$.
	When $n=1$, we have the magnitude of $x=x_s+x_d\epsilon\in\mathbb{DC}$ as
	$$|x| = \left\{\begin{array}{ll}
		{|x_s|+\frac{x_s^*x_d}{|x_s|}\epsilon,}
		& \ \mathrm{if} \  x_s \not = 0,\\
		|x_d|\epsilon, & \ \mathrm{if} \  x_s  = 0.
	\end{array}\right.$$
	If $\vx_s \not = \0$, then we say that $\vx$ is appreciable.
	The unit vectors in ${\mathbb R}^n$ are denoted as $\ve_1, \cdots, \ve_n$.   They are also unit vectors of ${\mathbb {DC}}^n$.

	We say $\vx= (x_1, \cdots, x_n)^\top \in {\mathbb {DC}}^n$ is a unit vector if $\|\vx\|_2=1$, or equivalently, $\|\vx_s\|_2=1$ and $\vx_s^*\vx_d=0$.
	We can define the normalization of $\vx$ following Theorem~3.3 in  \cite{CQ23}. Let $\vy=\frac{\vx}{\|\vx\|_2}$ be its normalization vector. Then   if $\vx$ is appreciable, we have
	\begin{equation*}
		\vy_s = \frac{\vx_s}{\|\vx_s\|_2}, \
		\vy_d = \frac{\vx_d}{\|\vx_s\|_2}-\frac{\vx_s}{\|\vx_s\|_2} \frac{\vx_s^*\vx_d}{\|\vx_s\|_2^2};
	\end{equation*}
	otherwise, if $\vx_s=\0$, there is
	\begin{equation*}
		\vy_s = \frac{\vx_d}{\|\vx_d\|_2}, \
		\vy_d \text{ is any complex vector satisfying }\vy_s^*\vy_d=0.
	\end{equation*}

	{A dual complex number matrix is denoted by $A=A_s+A_d\epsilon\in {\mathbb {DC}}^{n\times n}$.
		If $A_s$ is {invertible},
		then $A$ is also {invertible} and $A^{-1}=A_s^{-1}-A_s^{-1}A_dA_s^{-1}\epsilon$ \cite{QC23}.}
	The $F^R$-norm of $A\in\mathbb{DC}^{m\times n}$ is $\|A\|_{F^R}=\sqrt{\|A_s\|_F^2+\|A_d\|_F^2}$.
	


	\bigskip
	
	\subsection{Eigenvalues of Dual Complex Matrices}
	
	The eigenvalues of a real matrix may be a complex number.    Thus, classical matrix analysis is conducted for complex matrices \cite{HJ12}.
	For eigenvalues of dual number matrices, we have to consider eigenvalues of dual complex matrices.    Eigenvalues of dual complex matrices were introduced in \cite{QL21}, and studied in detail in \cite{QC23}.
	
	Let $A\in\mathbb{DC}^{n\times n}$.
	If
	\begin{equation} \label{en1}
		A\vx = \lambda\vx,
	\end{equation}
	where $\vx$ is appreciable, i.e., $\vx_s \not = \0$, then $\lambda$ is called an eigenvalue of $A$, with an eigenvector $\vx$.

	Since $A = A_s + A_d\epsilon$, $\lambda = \lambda_s + \lambda_d \epsilon$, and $\vx = \vx_s + \vx_d \epsilon$,
	(\ref{en1}) is equivalent to
	\begin{equation} \label{en2}
		A_s\vx_s = \lambda_s\vx_s,
	\end{equation}
	with $\vx_s \not = \0$, i.e., $\lambda_s$ is an eigenvalue of $A_s$ with an eigenvector $\vx_s$, and
	\begin{equation} \label{en3}
		(A_s-\lambda_sI)\vx_d - \lambda_d\vx_s = -A_d\vx_s.
	\end{equation}
	
	Denote the spectral radius of $A_s$ by ${\rho(A_s)}$.  Then we see that for any eigenvalue $\lambda$ of $A$ and any positive number $\delta$, we have
	\begin{equation} \label{upperbound}
		|\lambda| < {\rho(A_s)} + \delta.
	\end{equation}
	As shown in \cite{QC23}, a square dual number matrix may have no eigenvalue at all, or have infinitely many eigenvalues.
	
	See the following example from
	\cite{QC23}.
	
	{\bf Example 1} - A dual number matrix $A$ has no eigenvalue at all.
	
	Suppose that $A = A_s+A_d\epsilon$, where
	\begin{equation*}
		A_s = \left(\begin{array}{cc}
			1 & 1 \\
			0 & 1
		\end{array} \right)
		\text{ and }
		A_d = \left(\begin{array}{cc}
			0 & 0 \\
			1 & 0
		\end{array} \right).
	\end{equation*}
	All possible eigenpairs of $A_s$ are: $\lambda_s=1$, $\vx_s= \left({\alpha \atop 0}\right)$, where $\alpha \not = 0$. Then (\ref{en3}) is equivalent to
	\begin{equation*}
		\left(\begin{array}{c}
			x_{d,2}\\
			0
		\end{array} \right) =\alpha \left(\begin{array}{c}
			\lambda_d\\
			-1
		\end{array} \right),
	\end{equation*}
	where $\alpha \not = 0$.
	Hence, (\ref{en3}) has no solution for all possible eigenpairs of $A_s$.    This implies that  $A$ has no eigenvalue at all.
	
	Note that $A$ in this example is a positive dual number matrix in the sense of \cite{QLY22}. However, $A$ has no eigenvalue at all.  Thus, it does not {make}  sense to discuss the Perron theorem for a general positive dual number matrix.  In this paper, we only consider the Perron theorem of a dual number matrix when its standard part is {primitive} or irreducible nonnegative.
	
	We have not defined the spectral radius of a square dual complex matrix $A$.   There are three cases for which it is difficult to define the spectral radius of $A$.  The first case is like Example 1, in which $A$ has no eigenvalue at all.  The second case is that though $A$ has some eigenvalues, but for all eigenvalues $\lambda_s$ of $A_s$, whose modulus is equal to the spectral radius of $A_s$, there is no complex number $\lambda_d$ such that $\lambda = \lambda_s + \lambda_d \epsilon$ is an eigenvalue of $A$.   The third case is that there is an eigenvalue $\lambda_s$ of $A_s$, whose modulus is equal to the spectral radius of $A$, and there {are}  infinitely many complex numbers $\lambda_d$, such that $\lambda = \lambda_s + \lambda_d \epsilon$ is an eigenvalue of $A$, and the {modulus} of such $\lambda_d$ is unbounded.  Example 2 of \cite{QC23} is an example of the third case. Fortunately, if $A$ is primitive or irreducible nonnegative, these three cases do not happen.

	\bigskip
	
	\section{Primitive Dual Number Matrices}

	Suppose that $A = A_s + A_d\epsilon$ is an $n \times n$ dual number matrix, where $n$ is a positive integer, $A_s$ is a real primitive matrix, i.e., $A_s \ge O$ and there is a positive integer {$m$} such that $A_s^{{m}}$ is positive,  and $A_d$ is a real matrix.  Then we call $A$ a {\bf primitive dual number matrix}.
	We call a dual number matrix with a positive standard part a {\bf positive dual number matrix}.  {The dual transition probability matrix is such an example if its standard part is primitive.}
	A primitive dual number matrix with {$m=1$} is a {positive dual number matrix.}
	Note that {a primitive dual number matrix}
	may not be a nonnegative dual number matrix in the sense of \cite{QLY22}, as some components of $A$ may have a zero standard part and a negative dual part.

	\begin{Thm}
		Suppose that $A = A_s + A_d\epsilon$ is an $n \times n$  primitive  dual number matrix, where $n$ is a positive integer, $A_s$ is a real  primitive  matrix and $A_d$ is a real matrix.   Then $A$ has a positive dual number eigenvalue $\lambda = \lambda + \lambda_d\epsilon$, and an $n$-dimensional positive dual number  eigenvector $\vx = \vx_s + \vx_d\epsilon$, corresponding to $\lambda$.  Furthermore, $\lambda_s, \vx_s, \lambda_d$ and $\vx_d$ have the following properties.
		
		(a) $\lambda_s$ is the Perron eigenvalue of $A_s$,  with multiplicity $1$. For any other dual complex eigenvalue $\mu = \mu_s + \mu_d\epsilon$ of $A$, we have
		\begin{equation} \label{e7}
			\lambda_s > |\mu_s|.
		\end{equation}
		
		(b) $\vx_s$ is the Perron eigenvector of $A_s$.  For any eigenvector $\vz = \vz_s + \vz_d\epsilon$, corresponding to the other eigenvalue $\mu$ of $A$, $\vz_s$ cannot be a nonnegative vector.
		
		(c) We have
		\begin{equation} \label{e6aa}
			\lambda_d = {\vy_s^\top A_d \vx_s \over \vy_s^\top \vx_s},
		\end{equation}
		where $\vy_s$ is the Perron eigenvector of $A_s^\top$, with $\vy_s^\top \vx_s > 0$.   If $A_d$ is nonnegative, then {$\lambda_d$} is nonnegative.  If $A_d$ is nonnegative and nonzero, then {$\lambda_d$} is positive.
		
		(d) $\vx_d$ is the solution of
		\begin{equation} \label{e6ab}
			(A_s-\lambda_sI)\vx_d = (\lambda_dI-A_d)\vx_s.
		\end{equation}
		If we require the $2$-norm of $\vx$ to be $1$, then $\vx_d$ is unique.
		
		Let $A = (a_{ij})$. Then we have
		\begin{equation}  \label{e7a}
			\min_i \sum_{j=1}^n a_{ij} \le \lambda \le \max_i \sum_{j=1}^n a_{ij}.
		\end{equation}
	\end{Thm}
	\begin{proof}
		Since $A_s$ is a  real primitive  matrix, by the {\bf primitive} matrix theory \cite{BP94}, $A_s$ has a positive real eigenvalue $\lambda_s$ with a positive real eigenvector $\vx_s$ such that for any  other  eigenvalue $\mu$ of $A_s$, we have $\lambda_s > |\mu_s|$, and the multiplicity of $\lambda_s$ is $1$.
		Hence, we may assume that $\lambda_s$ and $\vx_s$ in (\ref{en2}) are the positive real eigenvalue and its positive real eigenvector of $A_s$, with properties (a) and (b).
		
		
		Next, we show (\ref{en3}) has a solution $\lambda_d$ and $\vx_d$, and  derive the formula (\ref{e6aa}) for $\lambda_d$.
		We may write (\ref{en3}) as
		\begin{equation} \label{e11}
			B\vz = -A_d\vx_s,
		\end{equation}
		where $B$ is an $n \times (n+1)$ matrix with the first $n$ columns as $A_s-\lambda_s I$ and the last column as $\vx_s$, i.e.,
		$$B = \left(\begin{array}{cc}
			A_s - \lambda_sI & \vx_s
		\end{array}\right)\ {\rm and}\ \vz =\left(	\begin{array}{c}\vx_d \\ -\lambda_d\end{array}\right).$$
		Since $\lambda_s$ is an eigenvalue of $A_s$ with multiplicity of $1$  and   $\vx_s$ is its corresponding eigenvector, we have the  rank of $A_s - \lambda_s I$ is $n-1$.
		Furthermore,  let $\vy_s$ be the Perron eigenvector of $A_s^\top$ corresponding to $\lambda_s$. Then $\vy_s$ is also a positive real vector.
		By $\vx_s^\top \vy_s>0$ and $\vy_s^\top (A_s - \lambda_s I) = {\0}^\top$, the rank of $B$ is $n$ and \eqref{e11} always has a solution. Furthermore, (\ref{e6aa}) and (\ref{e6ab}) hold.  
		
		In addition, if we require the 2-norm of $\vx$ to be  1, then $\vx_s^T\vx_d=0$. Combining this with (\ref{en3}), we have
		\begin{equation} \label{equ:xd_unique}
			\left(
			\begin{array}{c}
				A_s - \lambda_sI   \\
				\vx_s^\top
			\end{array}\right)\vx_d =	\left(
			\begin{array}{c} \lambda_d\vx_s-A_d\vx_s\\0	\end{array}\right).
		\end{equation}
		Since the rank of $A_s - \lambda_sI$ is $n-1$ and the null space  of $(A_s - \lambda_sI)^\top$  is generated by $\vx_s$, we have the  rank of the coefficient matrix in \eqref{equ:xd_unique} is $n$.  Thus,   equation \eqref{equ:xd_unique} has a unique solution $\vx_d$.

		Now assume that $\mu = \mu_s + \mu_d\epsilon$ is an eigenvalue of $A$ with an eigenvector $\vv = \vv_s + \vv_d\epsilon$, such that $\mu_s \not = \lambda_s$.   Then  $\vv_s \not = \0$,  
		and $\mu_s$ is an eigenvalue of $A_s$ with an eigenvector $\vv_s$.  By the nonnegative matrix theory \cite{BP94}, (\ref{e7})  has to hold.

		Let $x_k = \min_j x_j$.   Then $x_k > 0$ {and has a positive appreciable part}.   From $A\vx = \lambda \vx$, we have
		$$\sum_{j=1}^n a_{kj} x_j = \lambda x_k,$$
		i.e.,
		$$\sum_{j=1}^n a_{kj} (x_j-x_k) + x_k\sum_{j=1}^n a_{kj} = \lambda x_k.$$
		This implies that
		$$\min_i \sum_{j=1}^n a_{ij} \le \sum_{j=1}^n a_{kj} \le \lambda.$$
		This proves the first inequality of (\ref{e7a}).   Similarly, we can prove the second inequality of (\ref{e7a}).
	\end{proof}

	{Note that it is nontrivial that $\lambda$ is a dual number, as in general, an eigenvalue of a dual number matrix is a dual complex number.}  We call $\lambda$ in this theorem the {\bf Perron eigenvalue} of $A$ and $\vx$ the {{\bf Perron eigenvector} or} {\bf right Perron eigenvector} of $A$.  By (\ref{e7}), for any other eigenvalue $\mu$ of $A$, we have $\lambda > |\mu|$.  Thus, we may call $\lambda$ the {\bf spectral radius} of $A$, denoted as $\rho(A) = \lambda$.
	
	{We may see that $A^\top = A_s^\top + A_d^\top\epsilon$ is also a primitive dual matrix.  The Perron eigenvalue $\lambda$ of $A$ is also the Perron eigenvalue of $A^\top$.   The Perron eigenvector $\vy = \vy_s + \vy_d\epsilon$ of $A^\top$ is called the {\bf left Perron eigenvector} of $A$.  We have $\vy^\top\vx > 0 $.}

	In the following, we extend the Collatz minimax theorem to primitive dual number matrices.
	
	
	\begin{Thm}\label{thm3.2}
		Suppose that $A = A_s + A_d\epsilon$ is an $n \times n$ dual number matrix, where $n$ is a positive integer, $A_s$ is a real  primitive  matrix and $A_d$ is a real matrix.  Let $\lambda = \lambda_s + \lambda_d \epsilon$ be the Perron eigenvalue of $A$.   Then for any $n$-dimensional positive dual number vector $\vx$, we have
		\begin{equation} \label{e12}
			\min_i {(A\vx)_i \over x_i} \le \lambda \le \max_i {(A\vx)_i \over x_i}.
		\end{equation}
		Furthermore, we have
		\begin{equation} \label{e12b}
			\max_{\vx>{\bf 0}}\min_i {(A\vx)_i \over x_i} = \lambda = \min_{\vx>{\bf 0}}\max_i {(A\vx)_i \over x_i}.
		\end{equation}
		
	\end{Thm}
	\begin{proof}  Let $\vx$ be an $n$-dimensional positive dual number vector and $X = {\rm diag}(x_1,$ $\cdots,$ $x_n)$.   Then the Perron eigenvalue of $A$ is equal to the Perron eigenvalue of $X^{-1}AX$.   Applying (\ref{e7a}) to $X^{-1}AX$, we have (\ref{e12}).
		Furthermore, by taking $\vx$ as the right Perron eigenvector  of $A$,  we have \eqref{e12b}.
	\end{proof}

	\bigskip


	\section{Dual Number Matrices {with}  Irreducible Nonnegative Standard Parts}
	
	The results in the last section can be extended to dual number matrices with irreducible nonnegative standard parts.
	
	\begin{Thm}
		Suppose that $A = A_s + A_d\epsilon$ is an $n \times n$  dual number matrix, where $n$ is a positive integer, $A_s$ is a real irreducible nonnegative matrix and $A_d$ is a real matrix.   Then $A$ has a positive dual number eigenvalue $\lambda = {\lambda_s} + \lambda_d\epsilon$, where $\lambda_s$ is a positive real number and $\lambda_d$ is a real number, and {an  $n$-dimensional} positive dual number eigenvector $\vx = \vx_s + \vx_d\epsilon$, where $\vx_s$ is a positive real $n$-dimensional vector and $\vx_d$ is a real $n$-dimensional vector, corresponding to $\lambda$, such that {$\lambda_s$ is an eigenvalue of $A_s$ with multiplicity $1$,} and for any other dual complex eigenvalue $\mu = \mu_s + \mu_d\epsilon$ of $A$, we have
		\begin{equation} \label{e7b}
			\lambda_s \ge |\mu_s|.
		\end{equation}
		{The formulas (\ref{e6aa}) and (\ref{e6ab}) still hold}.  {If we require the $2$-norm of $\vx$ to be $1$, then $\vx_d$ is also unique.}
		
		{Furthermore, letting $A = (a_{ij})$, then we have
			\begin{equation}  \label{e14}
				\min_i \sum_{j=1}^n a_{ij} \le \lambda \le \max_i \sum_{j=1}^n a_{ij}.
		\end{equation}}
	\end{Thm}
	\begin{proof} Since $A_s$ is a  real irreducible nonnegative matrix, by the irreducible nonnegative matrix theory \cite{BP94}, $A_s$ has a positive real eigenvalue $\lambda_s$ with a positive real eigenvector $\vx_s$ such that for any other eigenvalue $\mu$ of $A_s$, we have $\lambda_s \ge |\mu_s|$, and the multiplicity of $\lambda_s$ is $1$.   Then the other part of the proof is similar to the proof of Theorem 3.1.
	\end{proof}
	
	We call $\lambda$ in this theorem the {\bf Perron-Frobenius eigenvalue} of $A$ and $\vx$   the {\bf right Perron-Frobenius eigenvector}.    By the nonnegative matrix theory \cite{BP94, HJ12}, the irreducible nonnegative matrix $A_s$ has a period $h$, which is a positive integer.  If $h =1$, then for other eigenvalue $\mu_s$ of $A_s$, we have $\lambda_s > |\mu_s|$. This implies that for all other eigenvalues $\mu$ of $A$, we have $\lambda > |\mu|$.  Thus, $\lambda$ can be called the {\bf spectral radius} of $A$, and denoted as $\rho(A) = \lambda$.
	{By  Definition 2.10 in \cite{V62}, $h=1$ if and only if $A_s$ is also a primitive matrix.}
	Then $A$ is a primitive dual number matrix in this case.
	
	Suppose $h > 1$.  By {the} nonnegative matrix theory \cite{BP94, HJ12}, $A$ has $h-1$ eigenvalues $\lambda_{s, j} =\lambda_s e^{2\pi j \ii \over h}$, where $\ii$ is the imaginary unit, for $j = 1, \dots, h-1$.  These eigenvalues are all simple roots of $A_s$.  Thus, for each $\lambda_{s, j}$,
	there is {a unique} 
	complex number $\lambda_{d, j}$, such that $\lambda_j =
	\lambda_{s, j} + \lambda_{d, j}\epsilon$ is an eigenvalue
	of $A$.  For the other eigenvalues $\mu = \mu_s + \mu_d\epsilon$, we have $\lambda_s > |\mu_s|$, hence $\lambda > |\mu|$.  The question is,
	what is the relation of $\lambda$ and these $\lambda_j$ for $j = 1, \dots, h-1$? Or,  what is the relation of $\lambda_d$ and these $\lambda_{d, j}$ for $j = 1, \dots, h-1$?   The following example {shows} that there is no clear relation between $\lambda_d$ and these $\lambda_{d, j}$ for $j = 1, \dots, h-1$.
	
	{\bf Example 2} - A dual number matrix $A$, whose standard part $A_s$ is an irreducible nonnegative  matrix with a period $h$ greater than one.
	
	Suppose that $A = A_s+A_d\epsilon$, where
	\begin{equation*}
		A_s = \left(\begin{array}{cc}
			0 & 1 \\
			1 & 0
		\end{array} \right)
		\text{ and }
		A_d = \left(\begin{array}{cc}
			a & b \\
			c & d
		\end{array} \right).
	\end{equation*}
	Then $A_s$ is an irreducible nonnegative matrix, with its period $h =2$.  It has two eigenvalues: $\lambda_s=1$ with an eigenvector $\vx_s= \left({1 \atop 1}\right)$, and $\lambda_{s, 1}=-1$ with an eigenvector $\vx_s= \left({1 \atop -1}\right)$. Then
	we have $\lambda_d = {1 \over 2}(a+b+c+d)$ and $\lambda_{d, 1} = {1 \over 2}(a-b-c+d)$.   This implies that $A$ has two eigenvalues: the Perron-Frobenius eigenvalue $\lambda = \lambda_s + \lambda_d\epsilon$, and eigenvalue $\lambda_1 = \lambda_{s, 1} + \lambda_{d, 1}\epsilon$.  There is no clear relation between the modulus of the Perron-Frobenius eigenvalue $\lambda$, and the modulus of the other eigenvalue {$\lambda_1$}.
	
	We see that if $h > 1$, there are some troubles in distinguishing the largest eigenvalue of $A$.  Then we may use the ``shift'' technique to overcome this.   In the next section, we will discuss the ``shift'' technique.

	Similar  to Theorem 3.2,
	we have the following theorem.

	\begin{Thm}\label{thm4.2}
		Suppose that $A = A_s + A_d\epsilon$ is an $n \times n$ dual number matrix, where $n$ is a positive integer, $A_s$ is a real irreducible nonnegative  matrix and $A_d$ is a real matrix.  Let $\lambda = \lambda_s + \lambda_d \epsilon$ be the Perron-Frobenius eigenvalue of $A$.   Then for any $n$-dimensional positive dual number vector $\vx$, we have
		\begin{equation} \label{e15}
			\min_i {(A\vx)_i \over x_i} \le \lambda \le \max_i {(A\vx)_i \over x_i}.
		\end{equation}
		{Furthermore, we have
			\begin{equation} \label{e15b}
				\max_{\vx>{\bf 0}}\min_i {(A\vx)_i \over x_i} = \lambda = \min_{\vx>{\bf 0}}\max_i {(A\vx)_i \over x_i}.
			\end{equation}
		}
	\end{Thm}
	
	{The proof is similar to the proof of Theorem~\ref{thm3.2} and we do not repeat it here.}

	\bigskip

	\section{An Iterative Algorithm}
	In this section, we consider the Collatz method for computing the
	Perron eigenpair of a primitive dual number matrix and the  Perron-Frobenius eigenpair of a dual number matrix with  irreducible nonnegative standard parts.
	
	We first  present  some useful results.
	\begin{Lem}
		Suppose that $A = A_s + A_d\epsilon$ is an $n \times n$ dual number matrix, where $n$ is a positive integer, $A_s$ is a real primitive or irreducible nonnegative  matrix and $A_d$ is a real matrix.
		{Let $A_s=(a_{ijs})$.}
		Then the following results hold.
		\begin{itemize}
			\item[(i)] The summation of each row of $A$ is positive and appreciable, i.e.,
			\begin{equation*}
				\sum_{j=1}^n a_{ijs} >0,\ \forall\, 1\le i\le n.
			\end{equation*}
			
			\item[(ii)] For any vector $\vx=\vx_s+\vx_d\epsilon\in\mathbb {D}^{n\times 1}$ with $\vx_s>{\bf 0}$, we have
			\begin{equation*}
				A_s\vx_s >{\bf 0}.
			\end{equation*}
			
			\item[(iii)] For any vectors  $\vx,\vy\in\mathbb {D}^{n\times 1}$ with   $\vx\ge\vy$, we have
			\begin{equation*}
				A\vx\ge A\vy.
			\end{equation*}
			
		\end{itemize}
	\end{Lem}
	\begin{proof}
		The first two items follow  from the corresponding results for real matrices. Item~$(iii)$ follows from Theorem 1 in \cite{QLY22}.
	\end{proof}

	\begin{Thm}\label{thm5.1}
		Suppose that $A = A_s + A_d\epsilon$ is an $n \times n$ dual number matrix, where $n$ is a positive integer, $A_s$ is a real primitive or irreducible nonnegative  matrix and $A_d$ is a real matrix. Let $\vx^{(0)}$ be an arbitrary  dual number  vector with a positive standard part and $\vy^{(0)} = A\vx^{(0)}$.
		For $k=0,1,2,\dots$, define
		\begin{equation}\label{eqn:xkyk}
			\vx^{(k+1)} = \frac{\vy^{(k)}}{\|\vy^{(k)}\|_2},\ \vy^{(k+1)} = A\vx^{(k+1)},
		\end{equation}
		and
		\begin{equation}\label{eqn:lmdk}
			\underline{\lambda}_k = \min_i {(A\vx^{(k)})_i \over x_i^{(k)}},\ \text{and } \overline{\lambda}_k = \max_i {(A\vx^{(k)})_i \over x_i^{(k)}}.
		\end{equation}
		Assume that $\lambda$ is the Perron eigenvalue or Perron-Frobenius eigenvalue  of $A$. Then we have
		\begin{equation}\label{eqn:lmk_increasing}
			\underline{\lambda}_0\le \underline{\lambda}_1 \le \dots \le \lambda \le \dots\le\overline{\lambda}_1\le\overline{\lambda}_0.
		\end{equation}
		
	\end{Thm}
	\begin{proof}
		By Theorems~\ref{thm3.2} and \ref{thm4.2}, we have $ \underline{\lambda}_k \le \lambda \le  \overline{\lambda}_k$ for all $k\ge0$. We now prove that $\underline{\lambda}_k\le \underline{\lambda}_{k+1}$ and $\overline{\lambda}_{k+1}\le\overline{\lambda}_k$.
		For each $k\ge 0$ with $\vx_s^{(k)}>{\bf 0}$, we have
		\begin{equation*}
			\vy^{(k)} = A\vx^{(k)}\ge \underline{\lambda}_k\vx^{(k)}>{\bf 0} \text{ and } \vy_s^{(k)}>{\bf 0}.
		\end{equation*}
		Thus, $\vx_s^{(k+1)}>{\bf 0}$ and
		\begin{eqnarray*}
			A\vx^{(k+1)} = \frac{A\vy^{(k)}}{\|\vy^{(k)}\|_2}\ge\frac{\underline{\lambda}_kA\vx^{(k)}}{\|\vy^{(k)}\|_2} = \underline{\lambda}_k\frac{\vy^{(k)}}{\|\vy^{(k)}\|_2}=\underline{\lambda}_k\vx^{(k+1)}.
		\end{eqnarray*}
		Combining  this with \eqref{eqn:lmdk}, we have $\underline{\lambda}_k\le \underline{\lambda}_{k+1}$.
		Similarly, we can show $\overline{\lambda}_{k+1}\le\overline{\lambda}_k$. This completes the proof.
	\end{proof}

	\begin{algorithm}[t]
		\caption{The Collatz method {for computing eigenpairs of  dual number matrices with primitive and irreducible nonnegative standard parts}}
		\label{alg:collatz}
		\begin{algorithmic}[1]
			\REQUIRE   $A = A_s + A_d\epsilon\in\mathbb D^{n\times n}$ and $\vx^{(0)}$ satisfy Theorem \ref{thm5.1}. The maximal iteration $k_{\max}$, stopping criterion  $\delta_1,\delta_2>0$, and shift parameter $\rho>0$.
			\STATE Let $A=A+\rho I_n$, $\vy^{(0)} = A\vx^{(0)}$, and compute $\underline{\lambda}_{0}$ and $\overline{\lambda}_{0}$ by \eqref{eqn:lmdk}.  Denote \textsc{flag}=0.
			\FOR{$k=0,2,\dots, k_{\max}$,}
			\STATE Update $\vx^{(k+1)}$ and $\vy^{(k+1)}$ by \eqref{eqn:xkyk}.
			\STATE Update  $\underline{\lambda}_{k+1}$ and $\overline{\lambda}_{k+1}$ by \eqref{eqn:lmdk}.
			\IF{$\|\underline{\lambda}_{k+1}-\overline{\lambda}_{k+1}\|_{F^R}\le \|A\|_{F^R}\cdot\delta_1,$}
			\STATE Let \textsc{flag}=1 and stop.
			\ENDIF
			\ENDFOR
			{
				\STATE   Denote $\vx=\vx^{k+1}$, $\underline{\lambda}=\underline{\lambda}_{k+1}-\rho$, $\overline{\lambda}=\overline{\lambda}_{k+1}-\rho$,
				and $\lambda=\frac{\underline{\lambda}+\overline{\lambda}}{2}$.
				\IF{$k=k_{\max}$ and $|\underline{\lambda}_{s,k+1}-\overline{\lambda}_{s,k+1}|\le \|A\|_{F^R}\cdot\delta_2$,}
				\STATE Compute $\lambda_d$ and $\vx_d$ via  \eqref{e11}.
				\STATE  Let \textsc{flag}=2.
				\ENDIF
			}
		\end{algorithmic}
		\textbf{Output:} {\textsc{flag},} $\vx$, $\lambda$,
		$\underline{\lambda}$ and $\overline{\lambda}$.
	\end{algorithm}

	{For real matrices, the method described by (\ref{eqn:xkyk}) and (\ref{eqn:lmdk}) is usually called the Collatz method.}
	We summarize the Collatz method for computing   the
	Perron eigenpair of a primitive dual number  matrix and the  Perron-Frobenius eigenpair of  a dual number matrix with an  irreducible nonnegative standard part  in Algorithm~\ref{alg:collatz}.


	In Algorithm~\ref{alg:collatz}, $\rho>0$ is a shift parameter. If $A_s$ is irreducible nonnegative, then $A_s+\rho I_n$ is primitive.
	If $A_s$ is  weakly positive \cite{ZQX12}, i.e., $(A_s)_{ij}>0$ for all $i,j=1,\dots,n$ and $i\neq j$,
	then $A_s+\rho I_n$ is positive.
	See Figure~\ref{fig:A_s}.
	\begin{figure}[t]
		\centering
		\includegraphics[width=0.9\linewidth]{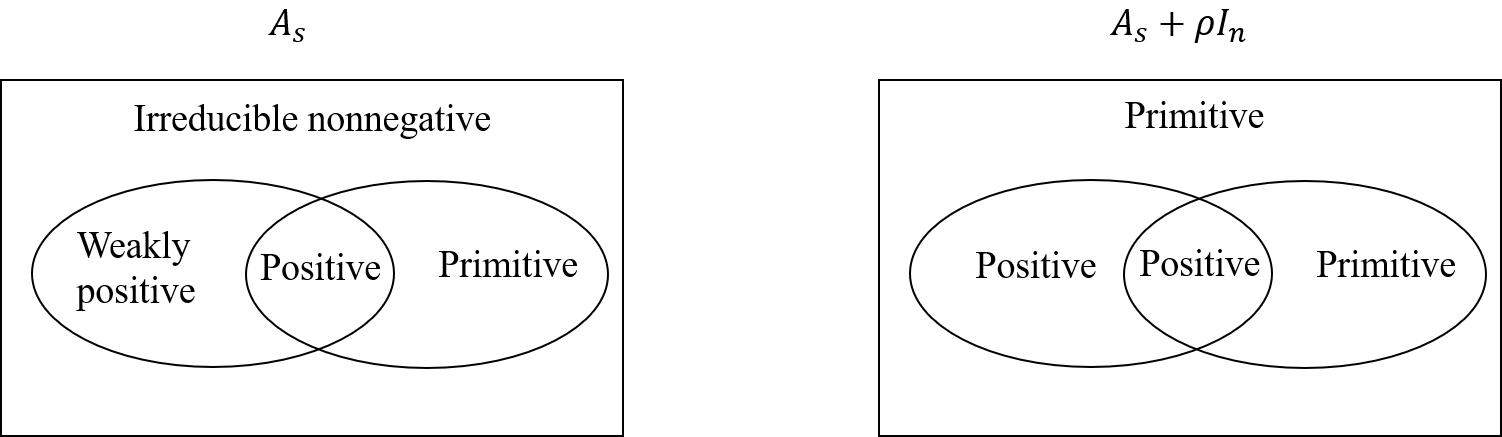}
		\caption{Relations of   irreducible nonnegative, weakly positive,  primitive, and positive  matrices and their shift.}
		\label{fig:A_s}
	\end{figure}

	In Algorithm~\ref{alg:collatz}, all computations of the standard parts in    $\vx_s^{(k)}$, $\vy_s^{(k)}$, $\underline{\lambda}_{s,k}$ and $\overline{\lambda}_{s,k}$
	are equivalent to that of implementing the Collatz method for $A_s$. If $A_d=O$,  $A$ reduces to a real matrix. Then  both  the sequences $\underline{\lambda}_{s,k}$ and $\overline{\lambda}_{s,k}$ converge to $\lambda_s$, from an arbitrary initial positive vector $\vx_s^{(0)}$, if and only if the
	irreducible matrix $A_s$ is primitive \cite{V62}.  Hence, {because the shift $\rho>0$,} both $\underline{\lambda}_{s,k}$ and $\overline{\lambda}_{s,k}$  converge to the same number $\lambda_s$.
	The convergence of dual number sequence  $\underline{\lambda}_{k}$ and $\overline{\lambda}_{k}$
	are more complicated.
	In together, we may categorize the results of Algorithm~\ref{alg:collatz} into three cases.
	\begin{enumerate}
		
		\item[(a)] \textsc{flag}=1. Both $\underline{\lambda}_{k}$ and $\overline{\lambda}_{k}$  converge to the same dual number $\lambda$, which is the Perron or Perron-Frobenius eigenvalue of $A$.
		
		{\bf Question.} What is the condition such that this case happens and what is the convergence rate?
		
		\item[(b)] \textsc{flag}=2. Both $\underline{\lambda}_{s,k}$ and $\overline{\lambda}_{s,k}$  converge to the same number $\lambda_s$, yet $\underline{\lambda}_{d,k}$ and $\overline{\lambda}_{d,k}$
		do not converge.
		In this case,   we compute the dual parts by solving the linear system \eqref{e11}.
		Since $B$  has full row rank,  the linear system \eqref{e11} is consistent.
		
		\item[(c)] \textsc{flag}=0. The algorithm stops because the  maximal iteration number is reached and $\underline{\lambda}_{s,k}$ and $\overline{\lambda}_{s,k}$ do  not converge. We can enlarge the maximal iteration  number  and implement Algorithm~\ref{alg:collatz} again.
	\end{enumerate}

	We show the convergence of $\{ \underline{\lambda}_{s,k} \}$ and $\left\{ \overline{\lambda}_{s,k} \right\}$ in the following theorem.
	\begin{Thm}\label{thm5.3}
		Suppose that $A = A_s + A_d\epsilon$ is an $n \times n$ dual number matrix, where $n$ is a positive integer, $A_s$ is a real primitive matrix, and $A_d$ is a real matrix. Let $\vx^{(0)}$ be the initial iterate, $\vx_s^{(0)}$ be the vector of  ones, $\rho>0$,
		and  $\lambda$ be the Perron eigenvalue  of $A$.  Let $\{ \underline{\lambda}_k \}$ and $\left\{ \overline{\lambda}_k \right\}$ be two dual number sequences generated by Algorithm 1.
		Then we have   $\lim_{k\rightarrow \infty}(\overline{\lambda}_{s,k}- \underline{\lambda}_{s,k})=0$. Furthermore, if $A_s$ is weakly positive, then
		\begin{equation}\label{con_rate}
			\overline{\lambda}_{s,k+1}- \underline{\lambda}_{s,k+1} \le \alpha(\overline{\lambda}_{s,k}- \underline{\lambda}_{s,k}),
		\end{equation}
		where $\alpha=1-\frac{\beta}{\overline{\mu}}$,   $\beta=\min\{\min\limits_{i,j=1,\dots,n,i\neq j}  a_{ijs}, \min\limits_{i=1,\dots,n} a_{iis}+\rho\}$, and $\overline{\mu}=\rho+\max\limits_{i=1,\dots,n}\sum_{j=1}^n a_{ijs}$.
	\end{Thm}
	\begin{proof}
		Since $A_s$ is a primitive matrix, we have   $\lim_{k\rightarrow \infty}(\lambda_s^{-1}A_s)^k=\vx_s\vy_s^\top$, where $\vx_s$ and $\vy_s$ are the right and the left Perron eigenvector of $A_s$, respectively.
		Therefore, we have
		$$\vx_s^{(k)}=\frac{A_s^k\vx_s^{(0)}}{\|A_s^k\vx_s^{(0)}\|_2}=\frac{(\lambda_s^{-1}A)^k\vx_s^{(0)}}{\|(\lambda_s^{-1}A)^k\vx_s^{(0)}\|_2}\rightarrow \frac{\vx_s}{\|\vx_s\|_2}.$$
		From this, we have $\lim_{k\rightarrow \infty}\overline{\lambda}_{s,k}=\lambda_s$ and $\lim_{k\rightarrow \infty}\underline{\lambda}_{s,k}=\lambda_s$.

		Furthermore, when $A_s$ is a weakly positive matrix, the result  in \eqref{con_rate} follows  directly from Theorem 4.1 in \cite{ZQX12} with $m=2$.
		
		This completes the proof.
	\end{proof}

	\section{Numerical Results}
	
	We show several numerical experiments for computing eigenpairs of dual number matrices with primitive and irreducible nonnegative standard parts.
	In Algorithm~\ref{alg:collatz}, we set $k_{\max}=2000$, $\delta_1=10^{-8}$, $\delta_2=10^{-8}$, and $\rho=1$ on default.
	
	{\subsection{Synthetic examples}}
	
	We first present several examples adopted from~\cite{ZQX12}.

	{\bf Example 5.1} Let $A_s$ satisfy $a_{1is}=a_{i1s}=1$ for $i=2,\dots,n$, and zero elsewhere, and $A_d$ be an $n\times n$ Jordan block with diagonal elements being one. In this example, $A_s$  is irreducible, but not primitive and weakly positive.
	
	{\bf Example 5.2} Let $A_s$ satisfy $a_{ijs}=i+j$ for $i,j=1,\dots,n$ and $i\neq j$, and zero elsewhere, and $A_d$ be an $n\times n$ Jordan block with diagonal elements being one. In this example,  $A_s$ is primitive and weakly positive, but not positive.
	
	{\bf Example 5.3} Let $A_s$ satisfy $a_{1ns}=1$, $a_{i1s}=1$ for $i=2,\dots,n$, $a_{nis}=1$ for $i=1,\dots,n-1$, and zero elsewhere, and $A_d$ be an $n\times n$ Jordan block with diagonal elements being one. In this example,  $A_s$  is primitive, but not weakly positive.
	
	{\bf Example 5.4} Let $A_s$ be a  random matrix where each element is  generated uniformly from $[0.1, 1.1]$, and $A_d$ be a Gaussian random matrix. In this example,  $A_s$  is positive.

	\begin{table}
		\centering
		\caption{Numerical results for Examples 5.1 to 5.4}
		\begin{tabular}{|c|c|cccc|}
			\hline
			Example &  $n$ & Eig & $\|A\vx-\lambda \vx\|_{F^R}$ & No. Iter & CPU time (s) \\ \hline
			& 10   &  $3.00+1.61\epsilon$  & 3.62e$-$11   &  32    &   2.08e$-$03 \\
			5.1 & 100   & $9.95+ 1.55 \epsilon$  & 2.86e-10  &   111  &   6.50e$-$03 \\
			& 1000  &  $ 31.61 +1.52\epsilon$  & 1.14e$-$09  &   355  &    4.08e$-$01 \\
			& 10000 &  $ 99.99+1.50\epsilon$   & 1.64e$-$09  &   1125   &   1.41e+02 \\
			\hline
			& 10  &  $ 103.62+1.89\epsilon$  & 2.73e-11  &   12    &   1.81e$-$03 \\
			5.2 & 100  &   $1.07\times 10^{5} + 1.99 \epsilon$ & 3.67e-12   &  8   &     1.60e$-$03 \\
			& 1000 &   $ 1.07\times 10^{7} +2.00\epsilon$  & 2.62e-11  &   8    &   1.37e$-$02 \\
			& 10000  &  $1.07\times 10^{9}+2.00 \epsilon$  & 2.47e-11  &   8   &    1.58e+00 \\
			\hline
			& 10   &  $2.17+ 1.58 \epsilon$ & 3.36e-11   &  32    &  2.29e$-$03 \\
			5.3 & 100  & $4.68+ 1.46 \epsilon$   & 1.17e-10   &  69    &   2.93e$-$03 \\
			& 1000 &   $10.03 +1.40 \epsilon$  & 3.86e-10  &   147   &    1.58e$-$01 \\
			& 10000  &  $ 21.56+1.36\epsilon$  & 9.14e-10  &   303   &    4.59e+01 \\
			\hline
			& 10      & $ 6.03+ 0.45 \epsilon$ &3.41e$-$09 &13.7   & 1.03e$-$03 \\
			5.4&100     & $ 59.98+ 0.13\epsilon$ & 4.67e$-$09 & 6.8   & 3.56e$-$04 \\
			&1000    &  $ 599.87+ 0.14\epsilon$ & 1.09e$-$08 & 4.7  &  1.02e$-$02 \\
			&10000    &  $ 5999.84 -0.33\epsilon$ & 2.38e$-$08&  3.7 &  8.63e$-$01 \\
			\hline
		\end{tabular}
		\label{tab:synth}
	\end{table}

	The numerical experiments for Examples 5.1-5.4 are shown in Table~\ref{tab:synth}.
	In Example 5.4, we repeat the random experiment 10 times and return the average performance.
	In this table, \textit{$n$} denotes the dimension of $A$,
	\textit{Eig} denotes the Perron-Frobenius eigenvalue,
	\textit{$\|A\vx-\lambda \vx\|_{F^R}$} denotes the residual in the eigenpair,
	\textit{No. Iter} denotes the number of
	iterations,
	\textit{CPU time (s)} denotes the CPU time  consumed in seconds.
	
	In all experiments, we can obtain the eigenpairs successfully with \textsc{Flag}=1 in Algorithm~\ref{alg:collatz}.
	This is corresponding to Theorem~\ref{thm5.3} that $\lim_{k\rightarrow \infty}(\overline{\lambda}_{s,k}- \underline{\lambda}_{s,k})=0$.
	We guess $\lim_{k\rightarrow \infty}(\overline{\lambda}_{d,k}- \underline{\lambda}_{d,k})=0$ also holds for primitive dual number matrices, though we do not have a proof at this moment.

	Examples 5.2 and 5.4 are weakly positive, and the iterative sequence converges very fast and stops in around 10 iterations.
	This is corresponding to Theorem~\ref{thm5.3} that $\overline{\lambda}_{s,k}- \underline{\lambda}_{s,k}$ converges to zero linearly.
	We guess $\overline{\lambda}_{k}- \underline{\lambda}_{k}$ also converges to zero linearly, and leave its proof for future study.

	\begin{figure}
		\centering
		\includegraphics[width=1\linewidth]{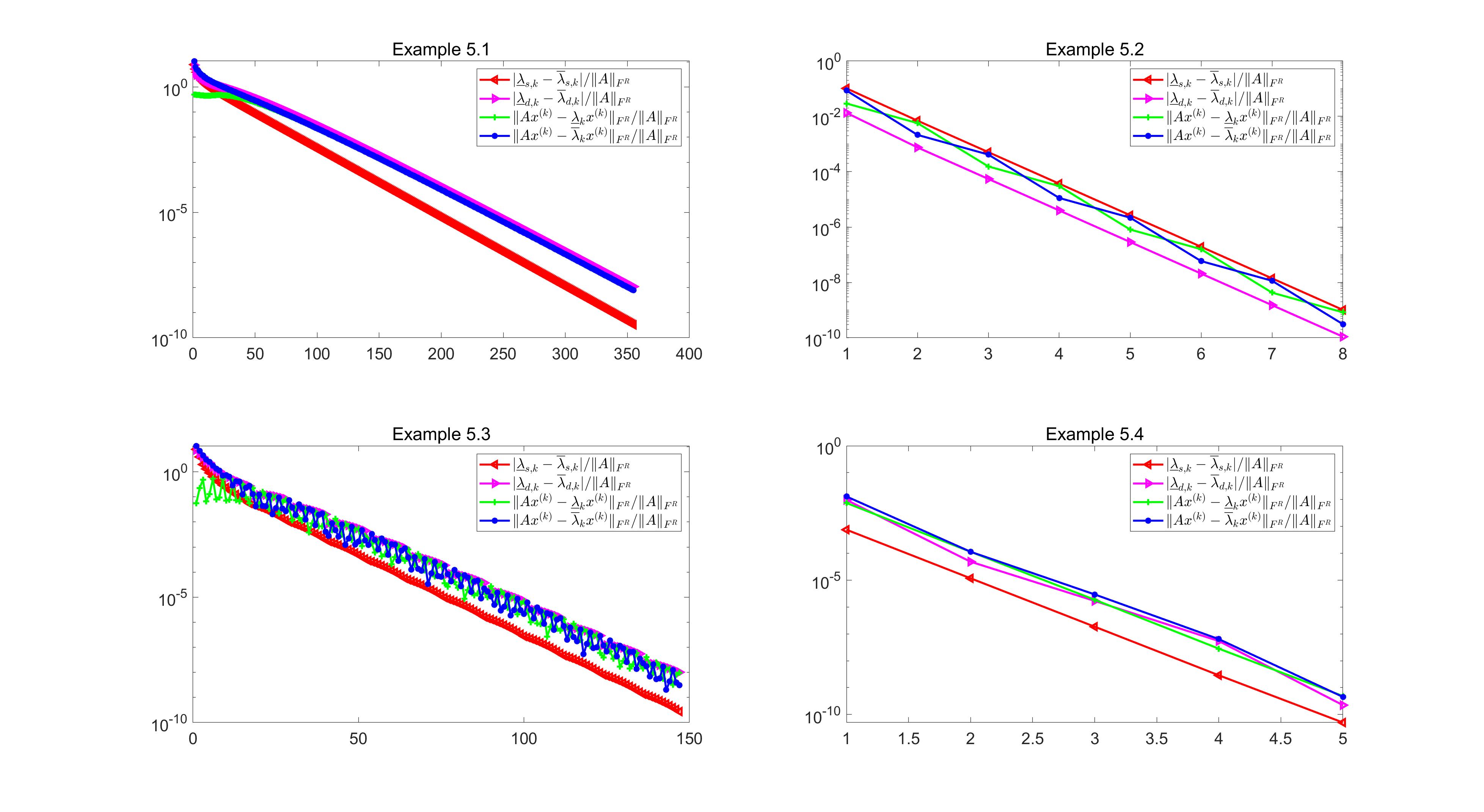}
		\caption{Iterative convergence results for Examples 5.1 to 5.4 with $n=1000$.}
		\label{fig:synth}
	\end{figure}
	
	{We further show the iterative convergence results for Examples 5.1 to 5.4 with $n=1000$ in Figure~\ref{fig:synth}.
		In this figure, we see that
		$\overline{\lambda}_{s,k}- \underline{\lambda}_{s,k}$,
		$\overline{\lambda}_{d,k}- \underline{\lambda}_{d,k}$,
		$\|A\vx^{(k)}-\underline{\lambda}_k\vx^{(k)}\|_{F^R}$, and $\|A\vx^{(k)}-\underline{\lambda}_k\vx^{(k)}\|_{F^R}$ all converge to zero almost linearly.
	}

	\subsection{Dual Markov chain}
	The dual Markov chain puts probability and perturbation in two parts.  This is different from the perturbed Markov chain \cite{M2005, LCN13}, which considers perturbed probability.
	
	\begin{Lem}\label{lem7.1}
		Suppose that $A = A_s + A_d\epsilon$ is an $n \times n$ dual number matrix, where $n$ is a positive integer, $A_s$ is a real primitive or irreducible nonnegative matrix, and $A_d$ is a real matrix.
		Assume that $\lambda$ and $\vx$ are  the Perron eigenvalue or Perron-Frobenius eigenvalue and eigenvector  of $A$, respectively.  Then we have the following results.
		\begin{itemize}
			\item[(i)]      Let $A(\theta) = A_s + \theta A_d\epsilon$ for all $\theta\in\mathbb R$, $\lambda(\theta)$ and $\vx(\theta)$ be  the Perron eigenvalue or Perron-Frobenius eigenvalue and eigenvector  of $A(\theta)$, respectively.  Then we have
			\begin{equation}
				\lambda(\theta) = \lambda_s+\theta \lambda_d \epsilon\text{ and } \vx(\theta)=\vx_s+\theta \vx_d{\epsilon}.
			\end{equation}
			\item[(ii)] Furthermore, assume $A_s^\top{\bf 1}={\bf 1}$ and $A_d^\top{\bf 1}={\bf 0}$. Then we have $\lambda_s=1$, and $\lambda_d=0$. Equivalently, we have
			$$\vx_s=A_s\vx_s\ \text{ and }\ \vx_d =  A_d\vx_s+A_s\vx_d.$$
		\end{itemize}
	\end{Lem}
	\begin{proof}
		Item~(i) follows directly from \eqref{en3}.
		
		(ii)  It follows from $A_s^\top{\bf 1}={\bf 1}$ that $\lambda_s=1$ and $\vy_s = {\bf 1}$.  Combing this with $A_d^\top{\bf 1}={\bf 0}$ and \eqref{e6aa} derives $\lambda_d=0$.
	\end{proof}
	
	\red{
		Next, we show the relationship between the dual Markov chain and the perturbed Markov chain by the perturbation analysis.
		
		\begin{Thm}\label{thm7.2}
			Suppose that $A = A_s + A_d\epsilon$ is an $n \times n$ dual number matrix, where $n$ is a positive integer, $A_s$ and $A_s+A_d$ are both real primitive or irreducible nonnegative matrices,  $A_s^\top{\bf 1}={\bf 1}$, and $A_d^\top{\bf 1}={\bf 0}$.
			Assume that $\lambda$ and $\vx$ are  the  Perron-Frobenius eigenvalue and eigenvector  of $A$, respectively, and $\vz=\vx_s+\delta\vz\in\mathbb R^{n\times 1}$ is the stationary state of  the perturbed transition matrix $P=A_s+ A_d$.
			Denote $\hat{\vx} = \vx_s+\vx_d$ and $(A_s-I_n)^D$ be the Drazin inverse of $A_s-I_n$.
			Then    $\vz-\hat{\vx}=\delta\vz - \vx_d$ and the following results hold.
			\begin{itemize}
				\item[(i)] If $\|(A_s-I_n)^DA_d\|_2<1$, then there exists $\vx_d$ such that
				\begin{equation}\label{equ:pmc-dmc:1}
					\frac{\|\vz-\hat{\vx}\|_2}{\| \vx_d\|_2}\le \frac{\kappa}{1-\|(A_s-I_n)^DA_d\|_2}\frac{\|A_d\|_2}{\|A_s-I_n\|_2},
				\end{equation}
				where $\kappa=\|A_s-I_n\|_2\|(A_s-I_n)^D\|_2$ is the condition number of group inverse of $A_s-I_n$.
				\item[(ii)] If $\|(A_s-I_n)^D\|_2^2\|A_d\|_2<1$, then for any $\vx_d$, we have
				\begin{equation}\label{equ:pmc-dmc}
					\frac{\|\vz-\hat{\vx}\|_2}{\| \vx_d\|_2}\le \frac{\|(A_s-I_n)^D\|_2^2\|A_d\|_2}{1-\|(A_s-I_n)^D\|_2^2\|A_d\|_2}.
				\end{equation}
			\end{itemize}
			
		\end{Thm}
		\begin{proof}
			From the definition, we have $P$ is also a transition matrix and
			$$\vx_s+\delta\vz =  (A_s+ A_d)(\vx_s+\delta\vz).$$
			By direct reformulations, we have $\delta\vz$ and  $ \vx_d$ satisfy  the following linear systems,
			\begin{equation*}
				(A_s-I_n+ A_d) \delta\vz =- A_d\vx_s \text{ and }  (A_s-I_n)\vx_d =- A_d\vx_s,
			\end{equation*}
			respectively.
			Because  $A_s$ and $A_s+A_d$ are both real primitive or irreducible nonnegative matrices,
			$\vx$ and $\vz$ exist.
			Hence, both linear systems above are consistent.  We may regard the first system as a perturbation of the second one.
			Since $A_s-I_n$ is singular, we apply the perturbation analysis of singular linear systems in \cite{Wei00, ZW03} here.
			An important notion  is the index of a singular matrix.
			Because $A_s-I_n$ has only one Jordan block corresponding to the zero eigenvalue  and the size of this Jordan block is one,  the index of $A_s-I_n$  is   one.

			Then item~(i) follows directly from Theorem~2.1 in \cite{Wei00}.
			
			(ii) Following   \cite{ZW03}, we multiply $A_s-I_n$ on   both hand sides of the above equations, and consider
			\begin{equation*}
				(A_s-I_n)(A_s-I_n+ A_d) \delta\vz =\vb \text{ and }  (A_s-I_n)^2\vx_d =\vb,
			\end{equation*}
			where $\vb=- (A_s-I_n)A_d\vx_s$.
			It follows from  $A_d^\top{\bf 1}=\0$ that $A_d$ is in the range space of $A_s-I_n$.
			Then by Theorem 4.1 in \cite{ZW03},   \eqref{equ:pmc-dmc} holds true.
			
			This completes the proof.
		\end{proof}
	}
	
	We first consider a   Markov chain with two states.
	Suppose that $A = A_s+A_d\epsilon$, with $A_s$ being the transition probability matrix and $A_d$ being its perturbation, where
	\begin{equation*}
		A_s = \left(\begin{array}{cc}
			0.5 & 0.5 \\
			0.5 & 0.5
		\end{array} \right)
		\text{ and }
		A_d = \left(\begin{array}{cc}
			0.3524 &  -0.0288 \\
			-0.3524 &   0.0288
		\end{array} \right).
	\end{equation*}
	For this example, we have $A_s^\top{\bf 1}={\bf 1}$ and $A_d^\top{\bf 1}={\bf 0}$.
	By direct computations and Lemma~\ref{lem7.1}, we have $\lambda=1$ and $\vx = \left(\begin{array}{c}
		0.7071 \\
		0.7071
	\end{array} \right)
	+ \left(\begin{array}{c}
		0.2288 \\
		-0.2288\end{array} \right)\epsilon$.
	
	{Denote the perturbed transition probability matrix by} $P(\theta)=A_s+\theta A_d$ {and its stationary state by}  $\vz(\theta)$ for $\theta\in[-1,1]$.  By direct computations, we have
	\begin{equation*}
		\vz(-1) = \left(\begin{array}{c}
			0.5272 \\
			0.8497
		\end{array} \right), \ \vz(0) = \left(\begin{array}{c}
			0.7071 \\
			0.7071
		\end{array} \right),
		\text{ and }
		\vz(1) = \left(\begin{array}{c}
			0.9543 \\
			0.2990
		\end{array} \right).
	\end{equation*}
	
	For dual Markov chain, we let $A(\theta)=A_s+\theta A_d \epsilon$ and compute  its Perron-Frobenius eigenpairs $\lambda(\theta)$ and $\vx(\theta)$. {By Lemma~\ref{lem7.1},}  we have
	\begin{equation*}
		\lambda(\theta)\equiv 1, \quad  \vx(\theta) = \left(\begin{array}{c}
			0.7071 \\
			0.7071
		\end{array} \right)
		+ \theta \left(\begin{array}{c}
			0.2288 \\
			-0.2288\end{array} \right)\epsilon.
	\end{equation*}
	
	\begin{figure}
		\centering
		\includegraphics[width=0.8\linewidth]{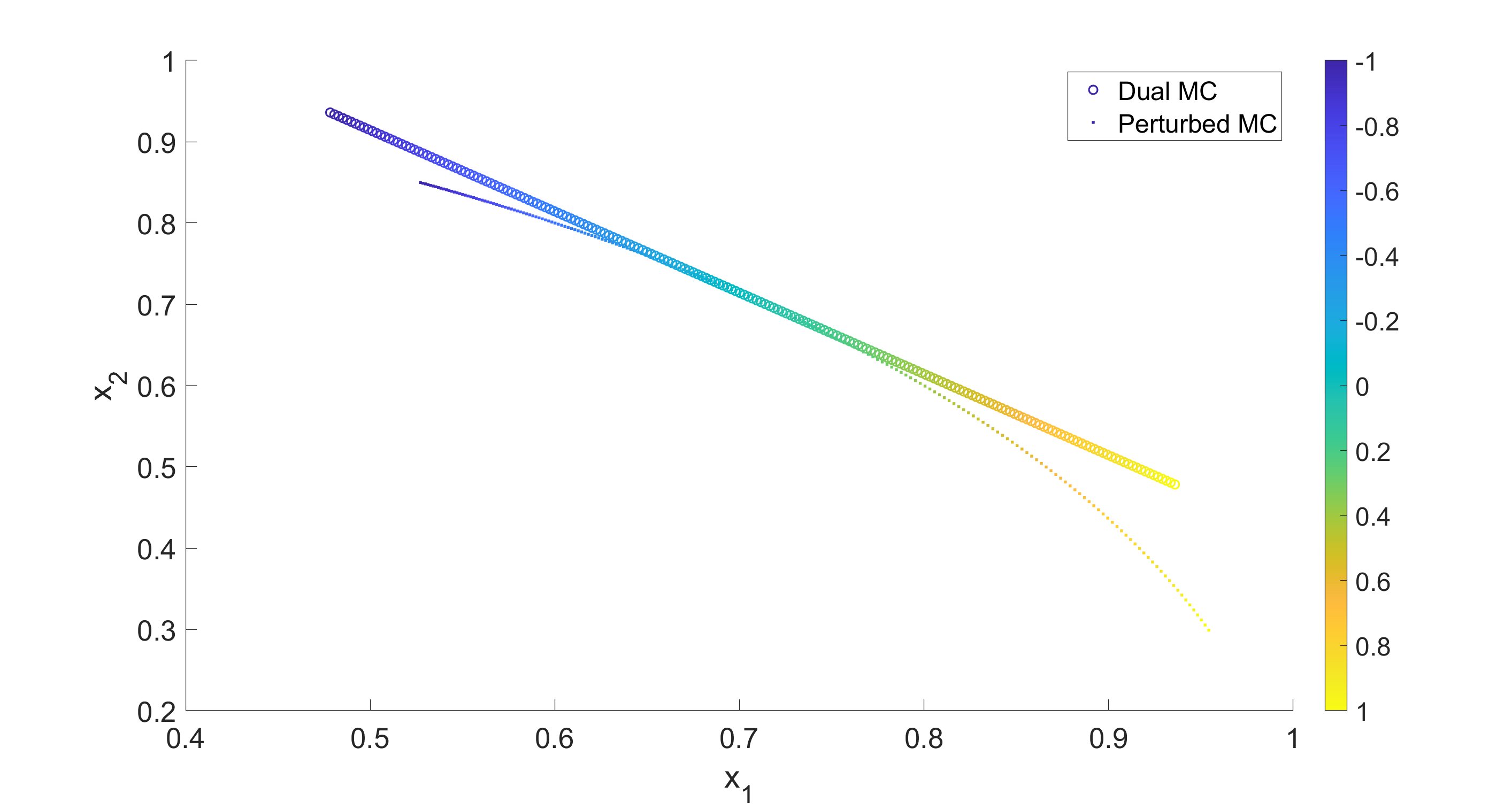}
		\caption{Comparisons of stationary states of dual Markov chain and  perturbed Markov chain by the 2-dimensional example. The colorbar denotes the value of $\theta$.}
		\label{fig:DMC}
	\end{figure}
	
	We plot $\vz(\theta)$ and $\hat{\vx}(\theta)=\vx_s+\theta\vx_d$ with $\theta\in[-1,1]$ in Figure~\ref{fig:DMC}.  From this figure, we see that   $\vz(\theta)$ and $\hat{\vx}(\theta)$ has similar perturbation {tendency}.
	The stationary states of the  perturbed Markov chain are nonlinear in $\theta$ while  the perturbed stationary vector from dual Markov chain is linear in $\theta$.
	When $\theta$ is small, the stationary states from the two methods are close to each other. This validates \red{the results in Theorem~\ref{thm7.2}}.
	
	\red{We continue to consider three real world transition matrices in \cite{CFN03,NQZ09}. A large soft-drink company in Hong Kong faces an in-house problem of production
		planning and inventory control.  All products are labelled as either very high sales volume (state 1), higher sales volume (state 2), standard sales
		volume (state 3), low sales volume (state 4) and very low sales volume (state 5).
		We first formulate the transition matrices from the categorical data sequences for the demands of   products A, B, and C   given in the appendix of \cite{CFN03}.
		Then we generate  the perturbation matrices $A_d$ as follows: we generate the first $n-1$ rows of $A_d$    by \texttt{randn},  compute the last row {such that $A_d^\top{\bf 1}={\bf 0}$}, and normalize $A_d$ such that $\|A_d\|_F=0.1$.
		The numerical distance of stationary states of perturbed Markov chain $\vz(\theta)$ and  dual Markov chain $\hat{\vx}(\theta)$ and their theoretical upper bound in  \eqref{equ:pmc-dmc}  are shown in Table~\ref{tab:DMC}.
		From this table, we see that the theoretical upper bound given in  \eqref{equ:pmc-dmc} is quite close to the numerical results.
	}
	
	At last, we consider  large-scale Markov Chains with random   transition probability matrices.
	Specifically, we generate  $A_s$ by \texttt{randi} in MATLAB and normalize each column such that its summation is equal to one. We generate the first $n-1$ rows of $A_d$    by \texttt{randn},  compute the last row {such that $A_d^\top{\bf 1}={\bf 0}$}, and normalize $A_d$ such that $\|A_d\|_F=0.5$.  We implement all experiments ten times and show their  average performance.
	\red{The results in Table~\ref{tab:DMC} are similar to that of the real transition matrix cases and validate the upper bound in  Theorem~\ref{thm7.2}.}
	

	\begin{table}[]
		\centering
		\caption{Numerical results on the distance  of stationary states between perturbed Markov chain $\vz(\theta)$ and dual Markov chain $\hat{\vx}(\theta)$.}
		\begin{tabular}{|c|cc|cc|cc|}\hline
			\multicolumn{7}{|c|}{Real transition matrices} \\ \hline
			&  \multicolumn{2}{|c|}{Product A} &
			\multicolumn{2}{|c|}{Product B} &
			\multicolumn{2}{|c|}{Product C} \\ \hline
			$\theta$	& $\|\vz(\theta)-\hat{\vx}(\theta)\|_2$  & \eqref{equ:pmc-dmc} &$\|\vz(\theta)-\hat{\vx}(\theta)\|_2$ &  \eqref{equ:pmc-dmc} & $\|\vz(\theta)-\hat{\vx}(\theta)\|_2$ &   \eqref{equ:pmc-dmc} \\\hline
			0.2 & 9.06e$-$05&  1.12e$-$03 & 4.39e$-$05 & 1.62e$-$04 & 9.06e$-$05&  1.12e$-$03 \\
			0.4 & 3.64e$-$04 & 4.97e$-$03 & 1.75e$-$04 & 6.69e$-$04 & 3.64e$-$04 & 4.97e$-$03  \\
			0.6 & 8.21e$-$04 & 1.25e$-$02 & 3.91e$-$04 & 1.56e$-$03 & 8.21e$-$04&  1.25e$-$02  \\
			0.8 & 1.47e$-$03 & 2.54e$-$02 & 6.92e$-$04 & 2.86e$-$03  &1.47e$-$03 & 2.54e$-$02  \\
			1.0 & 2.30e$-$03 & 4.60e$-$02 & 1.08e$-$03 & 4.64e$-$03 & 2.30e$-$03 & 4.60e$-$02  \\
			\hline
			\multicolumn{7}{|c|}{Random transition matrices} \\ \hline
			& \multicolumn{2}{|c|}{$n=10$} &\multicolumn{2}{|c|}{$n=100$} &\multicolumn{2}{|c|}{$n=1000$} \\ \hline
			$\theta$	& $\|\vz(\theta)-\hat{\vx}(\theta)\|_2$  & \eqref{equ:pmc-dmc} &$\|\vz(\theta)-\hat{\vx}(\theta)\|_2$ &  \eqref{equ:pmc-dmc} & $\|\vz(\theta)-\hat{\vx}(\theta)\|_2$ &   \eqref{equ:pmc-dmc} \\\hline
			0.2 & 1.15e$-$03 &  2.95e$-$03 & 8.66e$-$05 & 6.25e$-$04& 1.01e$-$05 & 2.59e$-$04  \\
			0.4 & 4.62e$-$03 & 1.38e$-$02 & 3.47e$-$04 &  2.77e$-$03  & 4.05e$-$05 & 1.13e$-$03  \\
			0.6 & 1.04e$-$02 & 3.73e$-$02 & 7.82e$-$04 & 6.99e$-$03 & 9.13e$-$05 & 2.78e$-$03  \\
			0.8 & 1.87e$-$02 & 8.31e$-$02 & 1.39e$-$03 &  1.42e$-$02 & 1.62e$-$04 & 5.46e$-$03  \\
			1.0 & 2.93e$-$02 & 1.74e$-$01 & 2.18e$-$03 & 2.57e$-$02 & 2.54e$-$04 & 9.54e$-$03  \\
			\hline
		\end{tabular}
		\label{tab:DMC}
	\end{table}

	\bigskip
	
	\section{Further Remarks}
	
	In this paper, we have extended the Perron-Frobenius theory to dual number matrices with primitive and irreducible {nonnegative} standard parts.  One motivation of our research is to consider probabilities as well as perturbation, or error bounds, or variances, in the Markov chain process.  We call such a Markov chain a dual Markov chain. A dual number matrix with a primitive or irreducible nonnegative standard part  always has a positive dual number eigenvalue with a positive dual number eigenvector.   If the standard part of the dual number matrix is primitive, then this positive eigenvalue is larger than the modulus of any other eigenvalue.  We call this positive eigenvalue the Perron eigenvalue of the dual number matrix then.   If the standard part of the dual number matrix is irreducible nonnegative but not primitive, then there are several other eigenvalues such that the {modulus} 
	of the standard part of each of these eigenvalues is equal to the standard part of this positive eigenvalue.   Then we call this positive eigenvalue the Perron-Frobenius eigenvalue of the dual number matrix.   The Collatz minimax theorem holds for such dual number matrices.  Based upon this, we construct an algorithm to compute this positive eigenvalue and its eigenvector.   In the case that the standard part of the dual number matrix is irreducible nonnegative but not primitive, we adopt a ``shift'' technique to distinguish this positive eigenvalue such that the sequence generated by the algorithm can be convergent.   Numerical results confirmed our theoretical discussion.
	\red{We also compared the dual Markov chain and the perturbed Markov chain by both theoretical upper bound and numerical experiments.}

	There are some questions that remain for future research.
	One question, assume that $A = A_s + A_d\epsilon$ is an $n \times n$ dual number matrix.   A conjecture is that $A^k$ tends to $O$ if and only if $\rho(A_s) < 1$.
	{This is true if $A_sA_d=A_dA_s$.}
	Is this true {for the general case}?  If it is, how to prove it?     In Section 6, there are also several convergence and convergence rate results, which are observed numerically, without theoretical proofs.
	
	\bigskip
	\red{
		{\bf Acknowledgment}  We are thankful to Professor Yimin Wei for discussion on perturbation analysis of singular linear systems, and Professor Lubin Cui for discussion on Markov chain.
	}
	
	

	\bigskip
	


\end{document}